\documentclass[11pt,a4paper]{amsart}
\usepackage[left=1in, right=1in,top=1in,bottom=1in]{geometry}
\usepackage{amsthm}
\usepackage{amsmath,amssymb,enumerate}
\usepackage{xcolor}
\usepackage{hyperref}
\newtheorem{theorem}{Theorem}

\newtheorem{prop}[theorem]{Proposition}
\newtheorem{lemma}[theorem]{Lemma}
\theoremstyle{remark}
\newtheorem{definition}{Definition}

\numberwithin{theorem}{section}

\allowdisplaybreaks

\newcommand{\be} {\begin{equation}}
\newcommand{\ee} {\end{equation}}
\newcommand{\bea} {\begin{eqnarray}}
\newcommand{\eea} {\end{eqnarray}}
\newcommand{\Bea} {\begin{eqnarray*}}
\newcommand{\Eea} {\end{eqnarray*}}

\def\R{{\mathbb R}}

\def\R{{\mathbb R}}

\numberwithin{equation}{section}

\newcommand{\Rn} {\mathbb{R}^N}  
\begin{document}

\title[$(p,q)$-fractional elliptic problem]{Critical $(p,q)$-fractional problems involving a sandwich type nonlinearity}

 \author[Bhakta]{Mousomi Bhakta}
\address{M. Bhakta, Department of Mathematics\\
Indian Institute of Science Education and Research Pune (IISER-Pune)\\
Dr Homi Bhabha Road, Pune-411008, India}
\email{mousomi@iiserpune.ac.in}
\author[Fiscella]{Alessio Fiscella}
\address{A. Fiscella, Departamento de Matem\'atica\\
Universidade Estadual de Campinas, IMECC, Rua S\'ergio Buarque de Holanda 651, CEP 13083--859,  Campinas SP, Brazil}
 \email{fiscella@unicamp.br}
  \author[Gupta]{Shilpa Gupta}
\address{S. Gupta, Department of Mathematics\\
Indian Institute of Science Education and Research Pune (IISER-Pune)\\
Dr Homi Bhabha Road, Pune-411008, India}
\email{shilpagupta890@gmail.com}

\keywords{nontrivial solution, positive solution, general open sets, fractional $(p,q)$ Laplacian, negative energy.} 
\subjclass[2010]{35J62, 35J70, 35R11, 35J20, 49J35}

\begin{abstract}
In this paper, we deal with the following $(p,q)$-fractional problem
$$
(-\Delta)^{s_{1}}_{p}u  +(-\Delta)^{s_{2}}_{q}u=\lambda P(x)|u|^{k-2}u+\theta|u|^{p_{s_{1}}^{*}-2}u \, \mbox{ in }\, \Omega,\qquad
u=0\, \mbox{ in }\,  \mathbb{R}^{N} \setminus \Omega,
$$
where $\Omega\subseteq\Rn$ is a general open set,  $0<s_{2}<s_{1}<1$,  $1<q<k<p<N/s_{1}$, parameter $\lambda,\ \theta>0$, $P$ is a nontrivial nonnegative weight, while $p_{s_{1}}^{*}=Np/(N-ps_{1})$ is the critical exponent. We prove that there exists a decreasing sequence $\{\theta_j\}_j$
 such that for any $j\in\mathbb N$ and with $\theta\in(0,\theta_j)$, there exist $\lambda_*$, $\lambda^*>0$ such that above problem admits at least $j$ distinct weak solutions with negative energy for any $\lambda\in (\lambda_*,\lambda^*)$. On the other hand, we show there exists $\overline{\lambda}>0$ such that for any $\lambda>\overline{\lambda}$, there exists $\theta^*=\theta^*(\lambda)>0$ such that the above problem admits a nonnegative weak solution with negative energy for any $\theta\in(0,\theta^*)$.\end{abstract}

\maketitle
\section{Introduction}
\setcounter{equation}{0}
In the present paper, we study the existence of weak solutions to the following nonlocal problem
\begin{equation}\label{1.1}
(-\Delta)^{s_{1}}_{p}u  +(-\Delta)^{s_{2}}_{q}u=\lambda P(x)|u|^{k-2}u+\theta|u|^{p_{s_{1}}^{*}-2}u \, \mbox{ in }\, \Omega,\quad
u=0 \, \mbox{ in } \,  \mathbb{R}^{N} \setminus \Omega,
\end{equation}
where $\Omega\subseteq\Rn$ is a general open set,  $0<s_{2}<s_{1}<1$,  $1<q<k<p<N/s_{1}$, while $p_{s_{1}}^{*}=Np/(N-ps_{1})$ is the critical exponent and parameters $\lambda,\, \theta>0$. Naturally, the condition $u=0$ in $\R^N\setminus\Omega$ disappears when $\Omega=\R^N$.
To prove the existence of weak solutions for problem \eqref{1.1}, we assume that the weight function $P:\Omega\rightarrow [0,\infty)$  is nontrivial and satisfies the following condition:
\begin{enumerate}
	\item[$(P_{0})$] $P\in  L^r(\Omega), \quad r=\frac{p_{s_{1}}^{*}}{p_{s_{1}}^{*}-k}$.
\end{enumerate}
The nonlocal fractional operators $(-\Delta)^{s_{1}}_{p}$ and $(-\Delta)^{s_{2}}_{q}$ are defined along any function $\varphi\in C_0^\infty(\mathbb R^N)$ by
\begin{align*}
(-\Delta)^{s_1}_{p}\varphi(x)=\lim\limits_{\varepsilon\rightarrow 0}\int_{R^{N}\backslash B_{\varepsilon}(x)}\dfrac{|\varphi(x)-\varphi(y)|^{p-2}(\varphi(x)-\varphi(y))}{|x-y|^{N+ps_1}}dy,\\
(-\Delta)^{s_2}_{q}\varphi(x)=\lim\limits_{\varepsilon\rightarrow 0}\int_{R^{N}\backslash B_{\varepsilon}(x)}\dfrac{|\varphi(x)-\varphi(y)|^{q-2}(\varphi(x)-\varphi(y))}{|x-y|^{N+qs_2}}dy,
\end{align*}
for $x\in\mathbb{R}^{N}$, where $B_\varepsilon(x)$ denotes the ball in $\mathbb{R}^N$ of radius $\varepsilon>0$ at the center $x\in\mathbb{R}^N$.

Recently, problem \eqref{1.1} has been studied in \cite{Bhakta} under a sublinear and a superlinear settings, namely with $k\in(1,q)$ and with $k\in(p,p_{s_1}^*)$ respectively. The aim of the present paper is to look for weak solutions of \eqref{1.1} with negative energy under a sandwich type situation, that is with $k\in(q,p)$. The sandwich type case for \eqref{1.1} is more interesting and delicate, since it is strictly related to the double $(p,q)$-growth of the main operator in \eqref{1.1}. The local version of \eqref{1.1} with $s_1=s_2=1$ has been studied in \cite{Baldelli,HoSim} under a sandwich type situation, mainly considering either $\Omega=\mathbb R^N$ or $\Omega$ bounded. In order to overcome the lack of compactness of energy functional, arising from the presence of a critical Sobolev term in \eqref{1.1}, the authors of \cite{Baldelli,HoSim} exploit a Lions' concentration compactness argument. Here, instead, we introduce a tricky step analysis which allows us to study the compactness property of energy functional of \eqref{1.1}, considering $\Omega\subseteq\mathbb R^N$ a general open set. We also recall \cite{DY}, where the authors studied the vectorial version of problem \eqref{1.1} in the bounded domain, namely the system
\begin{align}\label{system}
\begin{cases}
(-\Delta)^{s_{1}}_{p}u  +(-\Delta)^{s_{2}}_{q}u=\lambda|u|^{k-2}u+\theta\displaystyle\frac{\alpha}{\alpha+\beta}|u|^{\alpha-2}u|v|^\beta &\mbox{ in }\Omega,\\
(-\Delta)^{s_{1}}_{p}v  +(-\Delta)^{s_{2}}_{q}v=\mu|v|^{k-2}v+\theta\displaystyle\frac{\beta}{\alpha+\beta}|u|^\alpha|v|^{\beta-2}v &\mbox{ in }\Omega,\\
u=v=0 &\mbox{ in }  \mathbb{R}^{N} \setminus \Omega,
\end{cases}
\end{align}
with $\alpha>1$, $\beta>1$ and $\alpha+\beta=p_{s_{1}}^{*}$, under the sandwich situation $k\in(q,p)$. In \cite{DY}, the authors establish the existence of a mountain pass weak solution of \eqref{system} with positive energy, when $\lambda>\lambda^*$, $\mu>\mu^*$ and $\theta\in(0,\theta^*)$. However, we observe there is a conflict between the parameters given in \cite{DY}, since the thresholds 
$$\lambda^*=\lambda^*(\theta)>0,\quad \mu^*=\mu^*(\theta)>0\quad\mbox{and}\quad\theta^*=\theta^*(\lambda,\mu)>0$$
are interconnected. As far as we know, the existence of a weak solution for \eqref{1.1}, \eqref{system} with positive energy is still an open problem under the sandwich situation.

Now, we are ready to state the main results of the present paper.
\begin{theorem}\label{t1}
Let $\Omega\subseteq\Rn$ be an open set. Let $0<s_{2}<s_{1}<1$ and $1<q<k<p<N/s_{1}$. Assume that $P$ is nontrivial and satisfies $(P_{0})$.

Then, there exists a sequence $\{\theta_j\}_j$ with $\theta_j>\theta_{j+1}>0$, such that for any $j\in\mathbb N$ and with $\theta\in(0,\theta_j)$, there exist $\lambda_*$, $\lambda^*>0$ such that problem \eqref{1.1} admits at least $j$ distinct weak solutions with negative energy for any $\lambda\in (\lambda_*,\lambda^*)$.
 \end{theorem}

The proof of Theorem \ref{t1} relies on a careful combination of variational and topological tools, such as truncation techniques and genus theory, similarly to \cite[Theorem 1.1]{Bhakta} under the sublinear situation. However, the sandwich growth $k\in(q,p)$ does not allow us to provide the existence of infinitely many solutions for \eqref{1.1}. For this, as in the classical case \cite[Theorem 1]{Baldelli}, in Theorem \ref{t1} we can actually prove the existence of a finite number of solutions for \eqref{1.1}, at very least. Indeed, as in \cite[Theorem 1.1]{Bhakta}, we can construct a monotone non-decreasing sequence $\{c_j\}_j$ of critical values by Krasnoselskii's genus theory. However, we can guarantee that the values $c_1\leq c_2\leq\ldots\leq c_j$ are negative just when $\theta<\theta_j$ and $\lambda>\lambda_*$, with possibly $\theta_j\to0$ and $\lambda_*=\lambda_*(\theta_j)\to\infty$ as $j\to\infty$. This is the crucial difference with respect to the sublinear case \cite[Theorem 1.1]{Bhakta}, where the $\{c_j\}_j$ are all negative when $\lambda<\lambda^*$ is small. In any case, with respect to \cite[Theorem 1]{Baldelli} and \cite[Theorem 1.1]{Bhakta}, we are able to set our \eqref{1.1} on $\Omega\subseteq\mathbb R^N$ general, and not just $\Omega=\mathbb R^N$ or $\Omega$ bounded. 

In the next result, we study \eqref{1.1} whenever $\lambda$ is sufficiently large, a situation which is not covered by Theorem \ref{t1}.

\begin{theorem}\label{t3}
Let $\Omega\subseteq\Rn$ be an open set. Let $0<s_{2}<s_{1}<1$ and $1<q<k<p<N/s_{1}$. Assume that $P$ is nontrivial, nonnegative and satisfy $(P_{0})$.
	
Then, there exists $\overline{\lambda}>0$ such that for any $\lambda>\overline{\lambda}$, there exists $\theta^*=\theta^*(\lambda)>0$ such that problem \eqref{1.1} admits a nonnegative weak solution with negative energy for any $\theta\in(0,\theta^*)$.
\end{theorem}

The proof of Theorem \ref{t3} relies on a minimization argument combined with the Ekeland variational principle. Theorem \ref{t3} generalizes the local result stated in \cite[Theorem 1.1]{HoSim} in a fractional setting, allowing $\Omega\subseteq\Rn$ to be a general open set. Indeed, in \cite[Theorem 1.1]{HoSim} they mainly assume $\Omega$ bounded, in order to get the trivial intersection  $W_0^{1,p}(\Omega)\cap W_0^{1,q}(\Omega)=W_0^{1,p}(\Omega)$, with $q<p$. In this way, they minimize the energy functional on the ball
$$
B_\varrho(0)=\left\{u\in W_0^{1,p}(\Omega):\,\,\|\nabla u\|_p\leq \varrho\right\},
$$
that is neglecting the norm $\|\nabla u\|_q$. In Theorem \ref{t3}, we set our variational problem on a nontrivial intersection space $E$, endowed by norm $\|u\|_{E}=[u]_{s_{1},p}+[u]_{s_{2},q}$, as explained in Section \ref{sec2}.

The paper is organized as follows. In Section \ref{sec2}, we discuss the variational formulation of \eqref{1.1}, proving the compactness property of the energy functional. In Sections \ref{sec3} and \ref{sec4}, we prove Theorems \ref{t1} and \ref{t3}, respectively.

\section{Fractional Sobolev Spaces and Functional Framework}\label{sec2}

Let $\Omega$ be any open set of $\Rn$, let $s\in(0,1)$ and $m\in(1,\infty)$. We denote with $\|\cdot\|_{m}$ the norm of the space $L^{m}(\Omega)$.

We recall the fractional Sobolev space of parameters $s$ and $m$ as
$$
W^{s,m}(\Omega)=\left\{ u\in L^{m}(\Omega):\,\,\dfrac{|u(x)-u(y)|}{|x-y|^{\frac{N}{m}+s}}\in L^{m}(\Omega\times\Omega) \right\}.
$$
We define here the density space $Z^{s,m}(\Omega)$, given as the completion of $C_0^\infty(\Omega)$ with respect to the norm
 $$
[u]_{s,m}=\left( \iint_{\R^{2N}}\dfrac{|u(x)-u(y)|^{m}}{|x-y|^{N+sm}}dx dy\right) ^{1/m}.
$$
Even if $Z^{s,m}(\Omega)$ is not a real space of functions, but a density one, the choice of it is an improvement with respect to the space
$$
X_0^{s,m}(\Omega)=\left\{u\in W^{s,m}(\mathbb R^N):\,\,u=0\mbox{ a.e. in }\mathbb{R}^N\setminus
\Omega\right\},
$$
fairly popular in recent papers devoted to nonlocal variational problems.
Indeed, the density result proved in \cite[Theorem~6]{fsv} does not hold true for $X_0^{s,m}(\Omega)$, without assuming more restrictive conditions on the open
bounded set $\Omega$ and on its boundary $\partial\Omega$;
see in particular \cite[Remark 7]{fsv}. In conclusion, if $\Omega$ is an open bounded subset of $\mathbb{R}^N$, then $Z^{s,m}(\Omega)\subset X_0^{s,m}(\Omega)$, with possibly $Z^{s,m}(\Omega)\not= X_0^{s,m}(\Omega)$.

Furthermore, we recall the fractional Beppo-Levi space of parameters $s$ and $m$ as
\begin{equation}\label{bl}
D^{s,m}(\R^{N})=\left\{u\in L^{m_{s}^{*}}(\R^{N}):\,\, [u]_{s,m}<\infty\right\}
\end{equation}
with $m_{s}^{*}=Nm/(N-sm)$ when $N>sm$. It is worth noting that if $\Omega$ is any open subset of $\mathbb R^N$ and
$\widetilde u$ denotes the natural extension of any $u\in Z^{s,m}(\Omega)$, then
$\widetilde u\in D^{s,m}(\mathbb R^N)$. In other words,
\begin{equation}\label{notbl}
Z^{s,m}(\Omega)\subset\left\{u\in L^{m^*_s}(\Omega):\,\,\widetilde u\in D^{s,m}(\mathbb R^N)\right\}
\end{equation}
and equality holds when either $\Omega=\mathbb R^N$ or $\partial\Omega$ is continuous, as shown in \cite[Theorem~1.4.2.2]{G}. In what follows, with abuse of notation, we continue to write $u$ in place of $\widetilde u$, since the context is clear.

Due to the presence of a fractional $(p,q)$-Laplacian in problem \eqref{1.1}, the solution space is given by
$$
E=Z^{s_{1},p}(\Omega)\cap Z^{s_{2},q}(\Omega),
$$
which is a reflexive Banach space endowed with the norm
$$\|u\|_{E}=[u]_{s_{1},p}+[u]_{s_{2},q}.$$
We point out that when $\Omega$ is bounded, we can reduce $E=Z^{s_{1},p}(\Omega)$ thanks to \cite[Proposition 2.1]{valpal}, since $s_1>s_2$ and $p>q$. While, when $\Omega=\R^{N}$, we obtain that
$$
E=Z^{s_{1},p}(\R^N)\cap Z^{s_{2},q}(\R^N)=D^{s_{1},p}(\R^N)\cap D^{s_{2},q}(\R^N),
$$
thanks to \eqref{bl} and \eqref{notbl}.

Furthermore, in order to handle the critical Sobolev term in \eqref{1.1}, we set the fractional Sobolev constant as
\begin{equation}\label{a2}
S:=\inf_{u\in Z^{s_{1},p}(\Omega)\backslash\{0\}}\dfrac{[u]_{s_{1},p}^{p}}{\|u\|^{p}_{p_{s_{1}}^{*}}},
\end{equation}
which is strictly positive and well defined thanks to \cite[Theorem 5.6]{valpal}.

From now on, we assume that $0<s_{2}<s_{1}<1$, $1<q<k<p<N/s_{1}$, $(P_0)$ and $(Q_0)$ hold true, without further mentioning.
Thus, problem \eqref{1.1} has a variational structure with the energy functional $I_{\lambda,\theta}:E\rightarrow \R$ given by
\begin{equation*}
\begin{split}
I_{\lambda,\theta}(u)=\frac{1}{p}[u]_{s_{1},p}^{p}+\frac{1}{q}[u]_{s_{2},q}^{q}-\frac{\lambda}{k}\int_{\Omega}  P(x)|u|^{k}dx-\frac{\theta}{p_{s_{1}}^{*}}\|u\|^{p_{s_{1}}^{*}}_{p_{s_{1}}^{*}}.
\end{split}
\end{equation*}
The functional $I_{\lambda,\theta}$  is well defined and $I_{\lambda,\theta}\in C^{1}(E)$ following a similar argument as in \cite[Lemma 1 and Lemma 2]{Baldelli}.
 We see that $u\in E$ is a weak solution to \eqref{1.1} if and only if $\langle I'_{\lambda,\theta}(u),v\rangle=0$ for any $v\in E$, where
\begin{equation*}
\begin{split}
\langle I'_{\lambda,\theta}(u),v\rangle=\langle u,v\rangle_{s_1,p}+\langle u,v\rangle_{s_2,q}
-\lambda\int_{\Omega}  P(x)|u|^{k-2}uv dx-\theta\int_{\Omega}|u|^{p_{s_{1}}^{*}-2}u  v dx.
\end{split}
\end{equation*}
with
\begin{equation*}
\begin{split}
\langle u,v\rangle_{s_1,p}&=\iint_{\R^{2N}} \dfrac{|u(x)-u(y)|^{p-2}(u(x)-u(y))(v(x)-v(y))}{|x-y|^{N+ps_1}}dx dy,\\
\langle u,v\rangle_{s_2,q}&=\iint_{\R^{2N}} \dfrac{|u(x)-u(y)|^{q-2}(u(x)-u(y))(v(x)-v(y))}{|x-y|^{N+qs_2}}dx dy.
\end{split}
\end{equation*}
We first study the compactness property for $I_{\lambda,\theta}$, that is the Palais-Smale condition. For this, we say that sequence $\{u_{n}\}_n\subset E$ is a Palais-Smale sequence for $I_{\lambda,\theta}$ at level $c\in\R$ if
\begin{equation}\label{ps}
I_{\lambda,\theta}(u_n)\to c\,\mbox{ and }\,I'_{\lambda,\theta}(u_n)\to 0\,\mbox{ in }\,E^*\,\mbox{ as }\,n\to\infty.
\end{equation}
We say that $I_{\lambda,\theta}$ satisfies the Palais-Smale condition at a level $c\in\R$, $(PS)_{c}$ for short, if any Palais-Smale sequence $\{u_{n}\}_n$  admits a convergent subsequence.

Let us set $\overline{c}=\overline{c}(\lambda,\theta)$ as
\begin{equation}\label{cappello}
\overline{c}:=\left(\frac{1}{p}-\frac{1}{p_{s_1}^*}\right)\left[
	\frac{ S^{\frac{p_{s_1}^*}{p_{s_1}^*-p}}}{\theta^{\frac{p}{p_{s_1}^*-p}}}
	-\lambda^{\frac{p_{s_1}^*}{p_{s_1}^*-k}}\theta^{\frac{-k}{p_{s_1}^*-k}}\|P\|_{r}^{\frac{p_{s_1}^*}{p_{s_1}^*-k}}
	\cdot\frac{p_{s_1}^*-k}{k}\cdot\left(\frac{p-k}{p_{s_1}^*-p}\right)^{\frac{p_{s_1}^*}{p_{s_1}^*-k}}\right]
\end{equation}
with $S$ given in \eqref{a2}. We are able to prove the Palais-Smale condition under the threshold given in \eqref{cappello}.

\begin{lemma}\label{l5}
Let $\lambda>0$ and $\theta>0$. Then, functional $I_{\lambda,\theta}$ satisfies the $(PS)_{c}$ condition for any $c<\overline{c}$, with $\overline{c}$ such as in \eqref{cappello}.
\end{lemma}
 \begin{proof}
Fix $\lambda>0$, $\theta>0$ and let $\{u_{n}\}_n\subset E$ be a $(PS)_{c}$ sequence of $I_{\lambda,\theta}$. By $(P_0)$ and \eqref{a2}, we have
\begin{align*}
\begin{split}
I_{\lambda,\theta}(u_n)-\frac{1}{p_{s_{1}}^{*}}\langle I'_{\lambda,\theta}(u_n),u_n\rangle&=\left( \frac{1}{p}-\frac{1}{p_{s_{1}}^{*}}\right) [u_{n}]_{s_{1},p}^{p}+\left( \frac{1}{q}-\frac{1}{p_{s_{1}}^{*}}\right)[u_{n}]_{s_{2},q}^{q}\\
&\quad-\lambda\left( \frac{1}{k}-\frac{1}{p_{s_{1}}^{*}}\right)\int_\Omega P(x)|u_n|^kdx\\
&\geq\left( \frac{1}{p}-\frac{1}{p_{s_{1}}^{*}}\right) [u_{n}]_{s_{1},p}^{p}+\left( \frac{1}{q}-\frac{1}{p_{s_{1}}^{*}}\right)[u_{n}]_{s_{2},q}^{q}\\
&\quad-\lambda\left( \frac{1}{k}-\frac{1}{p_{s_{1}}^{*}}\right)S^{\frac{-k}{p}}\|P\|_{r}[u_{n}]_{s_{1},p}^{k}.
\end{split}
\end{align*}
Thus, by \eqref{ps} there exists $\beta_{\lambda,\theta}$ such that
\begin{align}\label{b1}
\begin{split}
   c+\beta_{\lambda,\theta}\|u_{n}\|_E+o(1)&\geq\left( \frac{1}{p}-\frac{1}{p_{s_{1}}^{*}}\right) [u_{n}]_{s_{1},p}^{p}+\left( \frac{1}{q}-\frac{1}{p_{s_{1}}^{*}}\right)[u_{n}]_{s_{2},q}^{q}\\
&\quad-\lambda\left( \frac{1}{k}-\frac{1}{p_{s_{1}}^{*}}\right)S^{\frac{-k}{p}}\|P\|_{r}[u_{n}]_{s_{1},p}^{k},
\end{split}
\end{align}
as $n\to\infty$.

We claim that $\{u_{n}\}_n$ is bounded in $E$. Assume for contradiction that $\|u_{n}\|_{E}\rightarrow \infty$. Then, passing if necessary to a subsequence, still labeled by $\{u_{n}\}_n$, at least one component has norm diverging as $n$ diverges. Hence we distinguish three situations.
\\

\begin{itemize}
\item \textbf{Case 1.} \emph{$[u_{n}]_{s_{1},p}\rightarrow \infty$ and $[u_{n}]_{s_{2},q}$ is bounded.}
\\

Dividing \eqref{b1} by $[u_{n}]_{s_{1},p}^{k}$, we get
$$
\frac{c}{[u_{n}]_{s_{1},p}^{k}}+\beta_{\lambda,\theta}\left(\frac{1}{[u_{n}]_{s_{1},p}^{k-1}}+\frac{[u_{n}]_{s_{2},q}}{[u_{n}]_{s_{1},p}^{k}}\right)+o(1)\geq \left(\frac{1}{p}-\frac{1}{p_{s_{1}}^{*}}\right) [u_{n}]_{s_{1},p}^{p-k}-\lambda\left( \frac{1}{k}-\frac{1}{p_{s_{1}}^{*}}\right)S^{\frac{-k}{p}}\|P\|_{r}
$$
which is a contradiction since $p>k$.

\vspace{0.1cm}
\item \textbf{Case 2.} \emph{$[u_{n}]_{s_{2},q}\rightarrow \infty$ and $[u_{n}]_{s_{1},p}$ is bounded.}
\\

Dividing \eqref{b1} by $[u_{n}]_{s_{2},q}^{q}$, we get
$$
\frac{c}{[u_{n}]_{s_{2},q}^{q}}+\beta_{\lambda,\theta}\left(\frac{[u_{n}]_{s_{1},p}^{p}}{[u_{n}]_{s_{2},q}^{q}}+\frac{1}{[u_{n}]_{s_{2},q}^{q-1}}\right)+o(1) \geq \left( \frac{1}{q}-\frac{1}{p_{s_{1}}^{*}}\right)-\lambda\left( \frac{1}{k}-\frac{1}{p_{s_{1}}^{*}}\right)S^{\frac{-k}{p}}\|P\|_{r}\frac{[u_{n}]_{s_{1},p}^{k}}{[u_{n}]_{s_{2},q}^{q}}
$$
which is a contradiction since $p_{s_{1}}^{*}>q$.

\vspace{0.1cm}
\item \textbf{Case 3.} \emph{$[u_{n}]_{s_{1},p}\rightarrow \infty$ and $[u_{n}]_{s_{2},q}\rightarrow \infty$.}
\\

We can rewrite \eqref{b1} as
\begin{align*}
\begin{split}
	c+\beta_{\lambda,\theta}[u_{n}]_{s_{2},q}-\left( \frac{1}{q}-\frac{1}{p_{s_{1}}^{*}}\right)[u_{n}]_{s_{2},q}^{q}+o(1)&\geq\left( \frac{1}{p}-\frac{1}{p_{s_{1}}^{*}}\right) [u_{n}]_{s_{1},p}^{p}\\
	&\quad-\lambda\left( \frac{1}{k}-\frac{1}{p_{s_{1}}^{*}}\right)S^{\frac{-k}{p}}\|P\|_{r}[u_{n}]_{s_{1},p}^{k}-\beta_{\lambda,\theta}[u_{n}]_{s_{1},p}
\end{split}
\end{align*}
which is a contradiction since, simultaneously, we have $q>1$ and $p>k>1$.
\end{itemize}

\noindent
Thus,  $\{u_{n}\}_n$ is bounded in $E$ and by the reflexivity of $E$, considering that $Z^{s_1,p}(\Omega)\hookrightarrow L^k(\Omega,P)$ compactly thanks to $(P_0)$ and  \cite[Lemma 4.1]{FPk}, there exists a subsequence, still denoted by $\{u_n\}_n$, and a function $u\in E$
such that
\begin{equation}\label{e2.4}
\begin{array}{ll}
u_n\rightharpoonup u\text{ in }E,\quad & u_n\rightharpoonup u\mbox{ in } L^{p_{s_1}^*}(\Omega), \\
u_n\rightarrow u\mbox{ in } L^{k}(\Omega,P), &u_n\to u\text{ a.e.  in }\Omega
\end{array}
\end{equation}
as $n\to\infty$.

Furthermore, as shown in the proof of \cite[Lemma~2.4]{CP},
by \eqref{e2.4} the sequence $\left\{\mathcal U_n\right\}_{n}$, defined in $\mathbb R^{2N}\setminus\mbox{Diag}\,\mathbb R^{2N}$ by
	\[(x,y)\mapsto\mathcal U_n(x,y)=\frac{|u_n(x)-u_n(y)|^{p-2}
(u_n(x)-u_n(y))}{|x-y|^{\frac{N+ps_1}{p'}}},
\]
is bounded in $L^{p'}(\mathbb R^{2N})$ as well as $\mathcal U_n\rightarrow\mathcal U$ a.e. in $\mathbb R^{2N}$, where
	\[\mathcal U(x,y)=\frac{|u(x)
-u(y)|^{p-2}(u(x)
-u(y))}{|x-y|^{\frac{N+ps_1}{p'}}}
\] and $p'$ is conjugate of $p$.
Thus, up to a subsequence, we get $\mathcal U_n\rightharpoonup\mathcal U$ in $L^{p'}(\mathbb R^{2N})$, and so as $n\to\infty$
\begin{equation}\label{e2.5}
\langle u_n,\varphi\rangle_{s_1,p}\rightarrow\langle u,\varphi\rangle_{s_1,p}
\end{equation}
for any $\varphi\in E$,
since $|\varphi(x)-\varphi(y)|\cdot|x-y|^{-\frac{N+ps_1}{p}}\in L^p(\mathbb R^{2N})$.
Similarly, we get as $n\to\infty$
\begin{equation}\label{e2.6}
\langle u_n,\varphi\rangle_{s_2,q}\rightarrow\langle u,\varphi\rangle_{s_2,q}\
\end{equation}
for any $\varphi\in E$.
Thanks to \eqref{e2.4}, using H\"older inequality it results
\begin{equation}\label{e2.8}
\lim_{n\to\infty}\int_{\Omega}P(x)|u_n|^{k-2}u_n(u_n-u)dx=0.
\end{equation}
Further, using \eqref{e2.4}, it's not difficult to check that
\begin{equation}\label{13-8-24-1}
|u_n|^{p^*_{s_1}-2}u_n \rightharpoonup |u|^{p^*_{s_1}-2}u \quad\mbox{ in }\,  L^{(p^*_{s_1})'}(\Omega).
\end{equation}
Consequently, using \eqref{ps}, \eqref{e2.4}--\eqref{13-8-24-1} we deduce that, as $n\to\infty$

\vspace{10mm}

\begin{align}\label{e2.9}
o(1)&=\langle I'_{\lambda,\theta}(u_n),u_n-u\rangle=[u_n]_{s_1,p}^p-\langle u_n,u\rangle_{s_1,p}
+[u_n]_{s_2,q}^q-\langle u_n,u\rangle_{s_2,q}\nonumber \\
&\quad-\lambda\int_{\Omega}P(x)|u_n|^{k-2}u_n(u_n-u)dx-\theta\int_{\Omega}|u_n|^{p_{s_1}^*-2}u_n(u_n-u)dx \nonumber \\
&=[u_n]^p_{s_1,p}-[u]^p_{s_1,p}+[u_n]^q_{s_2,q}-[u]^q_{s_2,q}-\theta\|u_n\|^{p_{s_1}^*}_{p_{s_1}^*}+\theta\|u\|^{p_{s_1}^*}_{p_{s_1}^*} +o(1).
\end{align}
Furthermore, by using \eqref{e2.4} and the celebrated Br\'ezis and Lieb Lemma in \cite{BL}, we have
\begin{align}\begin{split}\label{e2.10}
[u_n]^p_{s_1,p}-[u_n-u]^p_{s_1,p}&=[u]^p_{s_1,p}+o(1)\\
[u_n]^q_{s_2,q}-[u_n-u]^q_{s_2,q}&=[u]^q_{s_2,q}+o(1)\\
\|u_n\|^{p_{s_1}^*}_{p_{s_1}^*}-\|u_n-u\|^{p_{s_1}^*}_{p_{s_1}^*}
&=\|u\|^{p_{s_1}^*}_{p_{s_1}^*}+o(1),
\end{split}\end{align}
as $n\to\infty$.
Therefore, combining \eqref{e2.4} and relations \eqref{e2.9}--\eqref{e2.10},
we have proved the crucial formula
\begin{equation}\label{I}
\lim_{n\to\infty}[u_n-u]_{s_1,p}^p+\lim_{n\to\infty}[u_n-u]_{s_2,q}^q=\theta\lim_{n\to\infty}\|u_n-u\|^{p_{s_1}^*}_{p_{s_1}^*}.
\end{equation}
Now, denoting
$$
\lim_{n\to\infty}\|u_n-u\|_{p_{s_1}^*}=\ell,
$$
by \eqref{a2} and \eqref{I}, we have
$$
S\ell^p\leq\theta\ell^{p_{s_1}^*}
$$
from which, assuming by contradiction that $\ell>0$, we get
\begin{equation}\label{ll1}
\ell^{p_{s_1}^*-p}\ge \frac{S}{\theta}.
\end{equation}
On the other hand, by \eqref{ps} we have
\begin{align*}\begin{split}
c+o(1)&=I_{\lambda,\theta}(u_n)-\frac{1}{p}\langle I'_{\lambda,\theta}(u_n),u_n\rangle\\
&\geq
-\lambda\left(\frac{1}{k}-\frac{1}{p}\right)\int_\Omega P(x)|u_n|^kdx
+\theta\left(\frac{1}{p}-\frac{1}{p_{s_1}^*}\right)\|u_n\|^{p_{s_1}^*}_{p_{s_1}^*}
\end{split}\end{align*}
as $n\to\infty$, from which by \eqref{e2.4}, \eqref{e2.10} and H\"older's inequality, we get
\begin{align}\begin{split}\label{ll2}
c&\geq-\lambda\left(\frac{1}{k}-\frac{1}{p}\right)\int_\Omega P(x)|u|^kdx
+\left(\frac{1}{p}-\frac{1}{p_{s_1}^*}\right)\left(\theta\ell^{p_{s_1}^*}+\theta\|u\|^{p_{s_1}^*}_{p_{s_1}^*}\right)\\
&\geq-\lambda\left(\frac{1}{k}-\frac{1}{p}\right)\|P\|_{r}\left( \int_{\Omega}  |u|^{p_{s_{1}}^{*}}dx\right)^{\frac{k}{p_{s_{1}}^{*}}}
+\left(\frac{1}{p}-\frac{1}{p_{s_1}^*}\right)\left(\theta\ell^{p_{s_1}^*}+\theta\|u\|^{p_{s_1}^*}_{p_{s_1}^*}\right)\\
&:=f\left(\int_\Omega |u|^{p_{s_1}^*}dx\right)+\left(\frac{1}{p}-\frac{1}{p_{s_1}^*}\right)\theta\ell^{p_{s_1}^*},
\end{split}
\end{align}
where
\begin{align*}\begin{split}
&f(t):=a_1t-a_2t^{\frac{k}{p_{s_1}^*}},\quad\mbox{with}\\
&a_1:=\theta\left(\frac{1}{p}-\frac{1}{p_{s_1}^*}\right)\quad\mbox{and}\quad 
a_2:=\lambda\left(\frac{1}{k}-\frac{1}{p}\right)\|P\|_{r}.
\end{split}
\end{align*}
Since
$$
\min_{t\geq0}f(t)=f\left(\left(\frac{a_2\,k}{a_1\,p_{s_1}^*}\right)^{\frac{p_{s_1}^*}{p_{s_1}^*-k}}\right)=
-\frac{p_{s_1}^*-k}{p_{s_1}^*}\left(\frac{k}{p_{s_1}^*}\right)^{\frac{k}{p_{s_1}^*-k}}
a_2^{\frac{p_{s_1}^*}{p_{s_1}^*-k}}
a_1^{-\frac{k}{p_{s_1}^*-k}},
$$
from \eqref{ll2}, using also \eqref{ll1}, we obtain
$$
c\geq
\left(\frac{1}{p}-\frac{1}{p_{s_1}^*}\right)\left[
\frac{ S^{\frac{p_{s_1}^*}{p_{s_1}^*-p}}}{\theta^{\frac{p}{p_{s_1}^*-p}}}
-\lambda^{\frac{p_{s_1}^*}{p_{s_1}^*-k}}\theta^{\frac{-k}{p_{s_1}^*-k}}\|P\|_{r}^{\frac{p_{s_1}^*}{p_{s_1}^*-k}}
\cdot\frac{p_{s_1}^*-k}{k}\cdot\left(\frac{p-k}{p_{s_1}^*-p}\right)^{\frac{p_{s_1}^*}{p_{s_1}^*-k}}\right]
$$
which contradicts the hypothesis of the lemma $c<\overline{c}$, with $\overline{c}$ given in \eqref{cappello}.

Hence, $\ell=0$ which yields from \eqref{I} that $u_n\to u$ in $E$ as $n\to\infty$, concluding the proof.
\end{proof}

\section{The multiplicity result}\label{sec3}
 Because of the presence of a critical Sobolev term in \eqref{1.1}, we first need to truncate functional $I_{\lambda,\theta}$.
By using $(P_{0})$ and \eqref{a2}, we get
\begin{equation*}
\begin{split}
I_{\lambda,\theta}(u)&=\frac{1}{p}[u]_{s_{1},p}^{p}+\frac{1}{q}[u]_{s_{2},q}^{q}-\frac{\lambda}{k}\int_{\Omega}  P(x)|u|^{k}dx-\frac{\theta}{p_{s_{1}}^{*}}\|u\|^{p_{s_{1}}^{*}}_{p_{s_1}^*},\\
& \geq \frac{1}{p}[u]_{s_{1},p}^{p}-\frac{\lambda}{k}S^{\frac{-k}{p}}\|P\|_{r}[u]_{s_{1},p}^{k}-\frac{\theta}{p_{s_{1}}^{*}}S^{\frac{-p_{s_{1}}^{*}}{p}}[u]_{s_{1},p}^{p_{s_{1}}^{*}}.
\end{split}
\end{equation*}
Let us fix $\theta>0$ and let us define $g_{\lambda,\theta}:[0,\infty)\longrightarrow \R$ such that
\begin{align}\label{defg}
\begin{split}
&g_{\lambda,\theta}(t):=\frac{t^p}{p}-\lambda b_{1}t^{k}-b_{2}t^{p_{s_{1}}^{*}},\quad\mbox{with}\\
&b_{1}:=\frac{S^{\frac{-k}{p}}\|P\|_{r}}{k}\quad\mbox{and}\quad b_{2}:=\frac{\theta S^{\frac{-p_{s_{1}}^{*}}{p}}}{p_{s_{1}}^{*}}.
\end{split}
\end{align}
Now, we will check the location of critical points of $g_{\lambda,\theta}$. 
For this set 
$$g_{\lambda,\theta}(t)=t^{k}\,\widehat{g}_{\lambda,\theta}(t),\,\, \mbox{ where }\,\,\widehat{g}_{\lambda,\theta}(t)=-\lambda b_{1}+\frac{t^{p-k}}{p}-b_{2}t^{p_{s_{1}}^{*}-k}.$$
We have $\widehat{g}\,'_{\lambda,\theta}(t)>0$ for $t>0$ small enough, so that $\widehat{g}_{\lambda,\theta}$ is increasing near $0$, and 
\begin{equation}\label{s4.18}
t^{*}:=t^{*}(\lambda)=\left[\dfrac{(p-k)}{b_{2}\,p(p_{s_{1}}^{*}-k)} \right]^{1/(p_{s_{1}}^{*}-p)}>0
\end{equation}
is such that $\widehat{g}\,'_{\lambda,\theta}(t^{*})=0$.
Furthermore, $\widehat{g}_{\lambda,\theta}(0)<0$ and $\widehat{g}_{\lambda,\theta}(t)\longrightarrow -\infty$ as $t\longrightarrow \infty$. Thus, if $\widehat{g}_{\lambda,\theta}(t^{*})>0$ then there exist $t_{0}$, $t_{1}>0$ such that 
\begin{equation}\label{s4.10}
t^{*}\in(t_{0}, t_{1})\,\, \text{ and }\,\, \widehat{g}_{\lambda,\theta}(t_{0})=\widehat{g}_{\lambda,\theta}(t_{1})=0.
\end{equation}
For this, we observe that
$$\widehat{g}_{\lambda,\theta}(t^{*})=-\lambda\, b_{1}
+\frac{1}{p}\left[\dfrac{(p-k)}{b_{2}\,p(p_{s_{1}}^{*}-k)} \right]^{(p-k)/(p_{s_{1}}^{*}-p)}
-b_{2}\left[\dfrac{(p-k)}{b_{2}\,p(p_{s_{1}}^{*}-k)} \right]^{(p_{s_{1}}^{*}-k)/(p_{s_{1}}^{*}-p)}>0$$
if we assume $\lambda<\lambda^*$, with $\lambda^*=\lambda^*(\theta)$ given as
\begin{equation}\label{tc3}
\lambda^*:=\dfrac{1}{\theta^\frac{p-k}{p_{s_{1}}^{*}-p}}\frac{(p_{s_{1}}^{*}-p)k}{(p_{s_{1}}^{*}-k)p}
\left[\frac{(p-k)p_{s_{1}}^{*}}{(p_{s_{1}}^{*}-k)p}\right]^\frac{p-k}{p_{s_{1}}^{*}-p}
\frac{S^{\frac{p_{s_{1}}^{*}-k}{p_{s_{1}}^{*}-p}}}{\|P\|_{r}}.
\end{equation}

From the analysis above, if we consider $\lambda<\lambda^*$ then
$g_{\lambda,\theta}(t_{0})=g_{\lambda,\theta}(t_{1})=0$ and 
\begin{equation}\label{s4.9}
	g_{\lambda,\theta}(t)>0,\, \mbox{ for\, any }\, t\in(t_{0},t_{1})\, 
	\text{ and }\, g_{\lambda,\theta}(t)\leq 0 \, \mbox{ for\, any }\, t\in[0,t_{0}]\cup [t_{1},\infty).
	\end{equation}
Choose a cut off function $\eta\in C^{\infty}([0,\infty))$, non-increasing such that $\eta(t)=1$ if $0\leq t\leq t_{0}$ and $\eta(t)=0$ if $t\geq t_{1}$. 
We are ready to define the truncated functional corresponding to $I_{\lambda,\theta}$ as
\begin{equation}\label{tc1}
\widehat{I}_{\lambda,\theta}(u)=\frac{1}{p}[u]_{s_{1},p}^{p}+\frac{1}{q}[u]_{s_{2},q}^{q}-\frac{\lambda}{k}\int_{\Omega}  P(x)|u|^{k}dx-\eta([u]_{s_{1},p})\frac{\theta}{p_{s_{1}}^{*}}\|u\|^{p_{s_{1}}^{*}}_{p_{s_1}^*}.\end{equation}

Also, we define $\widetilde{g}_{\lambda,\theta}:[0,\infty)\rightarrow \R$ such that 
$$\widetilde{g}_{\lambda,\theta}(t):=\frac{t^{p}}{p}-\lambda b_{1}t^{k}-b_{2}\eta(t)t^{p_{s_{1}}^{*}},$$
where $b_{1}$ and $b_{2}$ are given in \eqref{defg}.
By definition of $\eta$, we have
\begin{equation}\label{s4.0}
 \widetilde{g}_{\lambda,\theta}(t)= g_{\lambda,\theta}(t) \,\mbox{ for\, any }\,   t\in(0, t_{0}),\quad \widetilde{g}_{\lambda,\theta}(t_{0})= g_{\lambda,\theta}(t_{0})=0
\end{equation}
and also
\begin{equation}\label{s4.1}
	\widetilde{g}_{\lambda,\theta}(t)\geq g_{\lambda,\theta}(t)>0\,\mbox{ for\, any }\, t\in(t_{0}, t_{1}), \quad \widetilde{g}_{\lambda,\theta}(t)>0 \,\mbox{ for\, any }\, t\in[t_{1},\infty).
\end{equation}
Furthermore, a straight-forward computation yields that
\begin{equation}\label{s4.12}
	\widehat{I}_{\lambda,\theta}(u)\geq \widetilde{g}_{\lambda,\theta}([u]_{s_{1},p})\, \mbox{ for\, any }\, u\in E\, \text{ and }\, 
	\widehat{I}_{\lambda,\theta}(u)=I_{\lambda,\theta}(u)\, \text{  for\, any }\, [u]_{s_{1},p}\in[0,t_{0}].
\end{equation}
We are now ready to study the properties of $\widehat{I}_{\lambda,\theta}$.

\begin{lemma}\label{l4.1}
Let $\theta>0$ and $\lambda\in(0,\lambda^*)$, with $\lambda^*$ as defined in \eqref{tc3}. Then:
\begin{itemize}
\item[$(i)$]  If $\widehat{I}_{\lambda,\theta}(u)<0$ then $[u]_{s_{1},p}<t_{0}$ and $I_{\lambda,\theta}(v)=\widehat{I}_{\lambda,\theta}(v)$ for any $v\in E$ in a small enough neighborhood of $u$. 
\item[$(ii)$]  $\widehat{I}_{\lambda,\theta}$ satisfies the $(PS)_{c}$ condition for any $c<0$.
\end{itemize}
\end{lemma}
\begin{proof}
Fix $\theta>0$ and $\lambda\in(0,\lambda^*)$.

\vspace{0.1cm}
\begin{itemize}
\item {$(i)$} Suppose that $\widehat{I}_{\lambda,\theta}(u)<0$.  By contradiction, we assume that $[u]_{s_{1},p}\geq t_{0}$. By \eqref{s4.1} and \eqref{s4.12}, we have
$\widehat{I}_{\lambda,\theta}(u)\geq \widetilde{g}_{\lambda,\theta}([u]_{s_{1},p})\geq 0$ which contradicts our assumption that $\widehat{I}_{\lambda,\theta}(u)<0$. Hence, $[u]_{s_{1},p}<t_{0}$. Thus, by \eqref{s4.12}, $I_{\lambda,\theta}(v)=\widehat{I}_{\lambda,\theta}(v)$ for all $v\in E$ in a small enough neighborhood of $u$.

\vspace{0.1cm}
\item {$(ii)$} Let $\{u_{n}\}_n$ be a $(PS)_{c}$ sequence for the functional $\widehat{I}_{\lambda,\theta}$ with $c<0$, that is satisfying
$$
\widehat{I}_{\lambda,\theta}(u_n)\to c\,\mbox{ and }\,\widehat{I}\,'_{\lambda,\theta}(u_n)\to 0\,\mbox{ in }\,E^*\,\mbox{ as }\,n\to\infty.
$$
Therefore, there exists $n_{0}\in\mathbb N$ large enough such that $\widehat{I}_{\lambda,\theta}(u_{n})<0$ for any $n>n_{0}$, consequently, $[u_{n}]_{s_{1},p}<t_{0}$. Using \eqref{s4.12}, we have $\widehat{I}_{\lambda,\theta}(u_{n})=I_{\lambda,\theta}(u_{n})$ and $\widehat{I}\,'_{\lambda,\theta}(u_{n})=I'_{\lambda,\theta}(u_{n})$, which imply that $\{u_{n}\}_n$ is a $(PS)_{c}$ sequence for the functional $I_{\lambda,\theta}$ with $c<0$.
Now, we observe that $\overline{c}$ given in \eqref{cappello} verifies $\overline{c}>0$ if $\lambda<\lambda^{**}$, with $\lambda^{**}=\lambda^{**}(\theta)$ given as
$$
\lambda^{**}:=\frac{p_{s_1}^*-p}{p-k}\left(\frac{k}{p_{s_1}^*-k}\right)^{\frac{p_{s_1}^*-k}{p_{s_1}^*}}
\frac{ S^{\frac{p_{s_1}^*-k}{p_{s_1}^*-p}}}{\|P\|_{r}\theta^\frac{p-k}{p_{s_{1}}^{*}-p}}.
$$
If we compare $\lambda^{**}$ with $\lambda^*$ given in \eqref{tc3}, we observe that
$$
\begin{aligned}
\lambda^*=&\dfrac{1}{\theta^\frac{p-k}{p_{s_{1}}^{*}-p}}\frac{(p_{s_{1}}^{*}-p)k}{(p_{s_{1}}^{*}-k)p}
\left[\frac{(p-k)p_{s_{1}}^{*}}{(p_{s_{1}}^{*}-k)p}\right]^\frac{p-k}{p_{s_{1}}^{*}-p}
\frac{S^{\frac{p_{s_{1}}^{*}-k}{p_{s_{1}}^{*}-p}}}{\|P\|_{r}} \\
&<\frac{p_{s_1}^*-p}{p-k}\left(\frac{k}{p_{s_1}^*-k}\right)^{\frac{p_{s_1}^*-k}{p_{s_1}^*}}
	\frac{ S^{\frac{p_{s_1}^*-k}{p_{s_1}^*-p}}}{\|P\|_{r}\theta^\frac{p-k}{p_{s_{1}}^{*}-p}}=\lambda^{**}
\end{aligned}
$$
if and only if
$$
\frac{p-k}{p}\left(\frac{k}{p_{s_1}^*-k}\right)^{\frac{k}{p_{s_1}^*}}\left[\frac{(p-k)p_{s_{1}}^{*}}{(p_{s_{1}}^{*}-k)p}\right]^\frac{p-k}{p_{s_{1}}^{*}-p}<1,
$$
which can be rewritten as
$$
\left(\frac{p-k}{p}\right)^{1-\frac{k}{p_{s_1}^*}}
\left[\frac{p-k}{p}\cdot\frac{k}{p_{s_1}^*-k}\cdot\frac{(p_{s_{1}}^{*}-k)p}{(p-k)p_{s_{1}}^{*}}\right]^{\frac{k}{p_{s_1}^*}}
\left[\frac{(p-k)p_{s_{1}}^{*}}{(p_{s_{1}}^{*}-k)p}\right]^{\frac{p-k}{p_{s_{1}}^{*}-p}+\frac{k}{p_{s_{1}}^*}}<1,
$$
that is
$$
\left(\frac{p-k}{p}\right)^{\frac{p_{s_1}^*-k}{p_{s_1}^*}}
\left(\frac{k}{p_{s_1}^*}\right)^{\frac{k}{p_{s_1}^*}}
\left[\frac{(p-k)p_{s_{1}}^{*}}{(p_{s_{1}}^{*}-k)p}\right]^{\frac{p-k}{p_{s_{1}}^{*}-p}+\frac{k}{p_{s_{1}}^*}}<1,
$$
which is trivially true, since $k<p<p_{s_1}^*$.

Hence, by Lemma \eqref{l5},  $I_{\lambda,\theta}$ satisfies the $(PS)_{c}$ condition for $c<0$, since $\lambda<\lambda^*<\lambda^{**}$ yields that $\overline{c}>0$.
\end{itemize}
\end{proof}

In order to prove the existence of multiple weak solutions for \eqref{1.1}, we need the genus theory. For this, we first recall the definition of genus.

\begin{definition}
	Let $X$ be a generic Banach space. Define,
	$$\Sigma=\{A\subset X\setminus\{0\}:\,\, A \text{ \,is\, closed\, and\, symmetric} \}.$$
	The genus  $\gamma(A)$ of $A\in\Sigma$	is defined as the smallest positive integer $d$ such that there exists an odd continuous map from $A$ to $\R^{d}\backslash\{0\}$.  If such $d$ does not exist, then we set $\gamma(A)=\infty$. Also, we define $\gamma(\emptyset)=0$
\end{definition}

The following proposition states the main properties on Krasnoselskii's genus which we need later.

\begin{prop}\cite[Lemma 1.2]{Ambrosetti}\label{s4p1}
	Let $A$, $B\in\Sigma$. Then the followings hold:
	\begin{enumerate}
		\item If there exists an odd continuous map from $A$ to $B$, then  $\gamma(A)\leq\gamma(B)$.
		\item If there exists $f\in C(A,X)$ odd, then $\gamma(A)=\gamma(f(A))$.
		\item If there exists an odd homeomorphism between $A$ and $B$, then $\gamma(A)=\gamma(B)$.
		\item If $S^{n-1}$ is the unit sphere in $\R^{n}$ then $\gamma(S^{n-1})=n$.
		\item $\gamma(A\cup B)\leq\gamma(A)+\gamma(B)$.
		\item If $\gamma(B)<\infty$, then $\gamma(\overline{A\backslash B})\geq \gamma(A)-\gamma(B)$.
		\item If $A$ is a  compact set, then $\gamma(A)<\infty$  and  there exists $\varepsilon>0$ such that $\gamma(A)=\gamma(N_{\varepsilon}(A))$, where $N_{\varepsilon}(A)=\{x\in X:d(x,X)\leq \varepsilon\}$.
		\item If $Y$ is a subspace of $X$ with codimension $k$ and $\gamma(A)>k$ then $A\cap Y\neq \emptyset$.
	\end{enumerate}
\end{prop}

Now, we are going to construct an appropriate minimax sequence of negative critical values for the truncated functional $\widehat{I}_{\lambda,\theta}$.
Considering $X=E$ as Banach space and using notation above, for any $j\in\mathbb N$ we define the minimax values $c_j=c_j(\lambda,\theta)$ as
\begin{equation}\label{cj}
	c_j:=\inf_{A\in \Sigma_j} \sup_{u\in A} \widehat{I}_{\lambda,\theta}(u),\quad\mbox{where }\,
	\Sigma_j:=\left\{A\in \Sigma:\,\,\gamma(A)\geq j\right\}.
\end{equation}
Clearly, $c_j \leq c_{j+1}$ for any $j\in\mathbb N$.

\begin{lemma}\label{l4.2}
Let $\theta>0$. There exists a sequence $\{\lambda_j\}_j\subset\mathbb R^+$, with $\lambda_j<\lambda_{j+1}$, such that for any $j\in\mathbb N$ we have
$$
-\infty<c_j<0
$$
for any $\lambda>\lambda_j$.
\end{lemma} 

\begin{proof}
Let us fix $\theta>0$.
Since $\widehat{I}_{\lambda,\theta}$ is trivially bounded from below, we get that $c_j>-\infty$ for any $j\in\mathbb N$.

Now, let us fix $j\in\mathbb N$. Since $P$ is nontrivial and nonnegative in $\Omega$, we can find an open set $\Omega_P\subseteq\Omega$ where $P(\cdot)>0$ on $\Omega_P$. Let us consider $\omega\subset \R^{N}$ a bounded open set having continuous boundary, with $|\omega|>0$, eventually with $\omega\subset\Omega_P$, that is such that  $P(\cdot)>0$ on $\omega$. Then, denoting with $\widetilde{u}$ the trivial extension of a generic function $u$ by zero, see  \cite{G}, and setting
$$\widetilde{Z}^{s_{1},p}(\omega)=\{u\in {Z}^{s_{1},p}(\omega):\,\, \widetilde{u}\in {Z}^{s_{1},p}(\R^{N}) \},$$
we get $\widetilde{Z}^{s_{1},p}(\omega)\subset E$ is a reflexive and separable space. Also, the map
$$
u\in\widetilde{Z}^{s_{1},p}(\omega)\,\,\mapsto\,\,\left( \int_\omega P(x)|u|^kdx\right)^{1/k},
$$
represents a norm.

Let $E_j$ be a $j$-dimensional subspace of $\widetilde{Z}^{s_{1},p}(\omega)$. Since all norms on $E_{j}$ are topologically equivalent, there exists $\delta_j>1$ such that
\begin{equation}\label{s4.3}
\delta_j^{-1}[u]_{s_2,q}\leq[u]_{s_1,p}\leq\delta_j\left( \int_\omega P(x)|u|^kdx\right)^{1/k},\quad\mbox{for\, any }\,u\in E_j.
\end{equation}
Without loss of generality, we can assume that $\delta_j<\delta_{j+1}$ for any $j\in\mathbb N$.

For any $t>0$ and any $u\in E_{j}$ with $[u]_{s_{1},p}=1$, by \eqref{s4.3} we get
\begin{equation}\label{s4.8}
\begin{aligned}
\widehat{I}_{\lambda,\theta}(tu)&=\frac{t^{p}}{p}+\frac{t^{q}}{q}[u]_{s_{2},q}^{q}-\frac{\lambda t^{k}}{k}\int_{\omega}  P(x)|u|^{k}dx-\eta\left(1\right)\frac{t^{p_{s_{1}}^{*}}}{p_{s_{1}}^{*}}\theta\int_{\omega}|u|^{p_{s_{1}}^{*}} dx\\
&\leq\frac{t^{p}}{p}+\frac{t^{q}}{q}\delta_j^q-\frac{\lambda t^{k}}{k}\delta_j^{-k}=t^{q}h_\lambda(t),
\end{aligned}
	\end{equation}
where
$$h_\lambda(t):=\frac{\delta_j^q}{q}-\lambda\frac{\delta_j^{-k}}{k}t^{k-q}+\frac{1}{p}t^{p-q}.$$
We observe that $h_\lambda(0)>0$ and $h_\lambda(t)\to\infty$ for $t\to \infty$. Also, since $p>k>q$ from
$$h'_\lambda(t)=-\lambda \frac{\delta_j^{-k}(k-q)}{k}t^{k-q-1}+\frac{p-q}{p}t^{p-q-1}=
t^{k-q-1}\left(-\lambda \frac{\delta_j^{-k}(k-q)}{k}+\frac{p-q}{p}t^{p-k}\right)
$$
we get $h'_\lambda(t)<0$ for $t\to0^+$ and $h'_\lambda(t)$ is increasing for $t>0$. Thus, we deduce there exists a unique $t^*_j>0$ such that 
$h'_\lambda(t^*_j)=0$ and $h_\lambda(t)\geq h_\lambda(t^*_j)$ for any $t\in(0,t^*_j)$, where $t^*_j=t^*_j(\lambda)$ is
$$t^*_j:=\left(\lambda\delta_j^{-k}\frac{p(k-q)}{k(p-q)} \right)^{\frac{1}{p-k}}>0.$$
Also, we observe that
\begin{align*}
h_\lambda(t^*_j)=\frac{\delta_j^q}{q}-\left(\lambda\delta_j^{-k}\right)^{\frac{p-q}{p-k}}\cdot\frac{p-k}{k(p-q)}\cdot\left[\frac{p(k-q)}{k(p-q)}\right]^{\frac{k-q}{p-k}},
\end{align*}
which implies that $h_\lambda(t^*_j)<0$ if and only if $\lambda>\lambda_j$, where
\begin{equation}\label{s4.6}
\lambda_{j}:=\delta_j^{\frac{q(p-k)}{p-q}+k}\cdot\frac{k(p-q)}{(p-k)^{\frac{p-k}{p-q}}(k-q)^{\frac{k-q}{p-q}}}\cdot q^{-\frac{p-k}{p-q}}\cdot p^{\frac{q-k}{p-q}}.
\end{equation}
Clearly, $\lambda_j<\lambda_{j+1}$ considering the monotonicity of $\{\delta_j\}_j$ given in \eqref{s4.3}.

Thus, for any $\lambda>\lambda_j$ and any $u\in E_{j}$ with $[u]_{s_{1},p}=1$, by \eqref{s4.8} we have
\begin{equation}\label{affermazione}
\widehat{I}_{\lambda,\theta}(t^*_ju)\leq (t^*_j)^qh_\lambda(t^*_j)=:-\varepsilon_j<0.
\end{equation} 
Let us set
$$
\mathbb{S}_j:=\left\{u\in E_j:\,\,[u]_{s_{1},p}=t^*_j\right\}.
$$
It is clear that $\mathbb{S}_j$ is homeomorphic to the $(j-1)$-dimensional sphere $S^{j-1}$. Hence, from Proposition \ref{s4p1} we know that $\gamma(\mathbb{S}_j)=j$. While, by \eqref{affermazione} we have that $\mathbb{S}_j\subset \widehat{I}_{\lambda,\theta}^{\,-\varepsilon_j}$, with
$$
\widehat{I}_{\lambda,\theta}^{\,-\varepsilon_j}:=\left\{u\in E:\,\,\widehat{I}_{\lambda,\theta}(u)\leq{-\varepsilon_j}\right\},
$$
and so by Proposition \ref{s4p1} again, we have
$$
\gamma(\widehat{I}_{\lambda,\theta}^{\,-\varepsilon_j})\geq\gamma(\mathbb{S}_j)=j.
$$
Since $\widehat{I}_{\lambda,\theta}(0)=0$ we have that $0\not\in \widehat{I}_{\lambda,\theta}^{\,-\varepsilon_j}$, so that considering $\widehat{I}_{\lambda,\theta}$ is even and continuous, we know that $\widehat{I}_{\lambda,\theta}^{\,-\varepsilon_j} \in\Sigma_j$. Finally, we get
$$
c_j= \inf_{A\in \Sigma_j} \sup_{u\in A} \widehat{I}_{\lambda,\theta}(u)\leq \sup_{u\in \widehat{I}_{\lambda,\theta}^{\,-\varepsilon_j}} \widehat{I}_{\lambda,\theta}(u)\leq -\varepsilon_j <0,
$$
concluding the proof.
\end{proof}

Clearly, once $j\in\mathbb N$ is fixed, the truncated functional $\widehat{I}_{\lambda,\theta}$ verifies Lemmas \ref{l4.1} and \ref{l4.2} if $\lambda\in(\lambda_j,\lambda^*)$, with $\lambda^*$ and $\lambda_j$ given in \eqref{tc3} and \eqref{s4.6}, respectively. In this direction, we observe that inequality $\lambda_j<\lambda^*$ can be rewritten as
$$
\delta_j^{\frac{q(p-k)}{p-q}+k}\cdot\frac{k(p-q)}{(p-k)^{\frac{p-k}{p-q}}(k-q)^{\frac{k-q}{p-q}}}\cdot q^{-\frac{p-k}{p-q}}\cdot p^{\frac{q-k}{p-q}}<
\dfrac{1}{\theta^\frac{p-k}{p_{s_{1}}^{*}-p}}\frac{(p_{s_{1}}^{*}-p)k}{(p_{s_{1}}^{*}-k)p}
	\left[\frac{(p-k)p_{s_{1}}^{*}}{(p_{s_{1}}^{*}-k)p}\right]^\frac{p-k}{p_{s_{1}}^{*}-p}
	\frac{S^{\frac{p_{s_{1}}^{*}-k}{p_{s_{1}}^{*}-p}}}{\|P\|_{r}},
$$
which holds true if $\theta<\theta_j$, with
\begin{equation}\label{qj}
\theta_j:=\delta_j^{-\frac{q(p_{s_{1}}^{*}-p)}{p-q}-\frac{k(p_{s_{1}}^{*}-p)}{p-k}}\cdot
\frac{(p-k)p_{s_{1}}^{*}}{(p_{s_{1}}^{*}-k)p}\cdot
\left(\frac{q}{p}\right)^{\frac{p_{s_{1}}^{*}-p}{p-q}}\cdot
\left[\frac{(p_{s_{1}}^{*}-p)(k-q)^{\frac{k-q}{p-q}}}{(p_{s_{1}}^{*}-k)(p-q)\|P\|_r}\right]^{\frac{p_{s_{1}}^{*}-p}{p-k}}\cdot S^{\frac{p_{s_{1}}^{*}-k}{p-k}}\cdot(p-k)^{\frac{p_{s_{1}}^{*}-p}{p-q}}.
\end{equation}
From the monotonicity of $\{\delta_j\}_j$, given in \eqref{s4.3}, we deduce that $\theta_j>\theta_{j+1}$ for any $j\in\mathbb N$. 

\begin{proof}[{\bf Proof of the Theorem \ref{t1}}]
Let $j\in\mathbb N$ be fixed and let us set $\lambda_j$ and $\theta_j$ as in \eqref{s4.6} and \eqref{qj}, respectively. Let $\theta\in(0,\theta_j)$, so that $\lambda_j<\lambda^*$, with $\lambda^*$ given in \eqref{tc3}. Then, let $\lambda\in(\lambda_j,\lambda^*)$ and let us consider the minimax sequence $\{c_j\}_j$ given in \eqref{cj}. By Lemma \ref{l4.2}, we know that
$$
-\infty<c_i<0,\quad\mbox{for\, any }\,i=1,\ldots,j,
$$
so that Lemma \ref{l4.1}$(ii)$ yields that $\widehat{I}_{\lambda,\theta}$ satisfies the $(PS)_{c_i}$ condition, for any $i=1,\ldots,j$.
Thus, $c_i$ are critical values for $\widehat{I}_{\lambda,\theta}$, see \cite{Rabinowitz}, and setting
$$
K_c:=\left\{u\in E\setminus\{0\}:\,\,\widehat{I}_{\lambda,\theta}(u)=c\,\mbox{ and }\,\widehat{I}\,'_{\lambda,\theta}(u)=0\right\},
$$
we infer that $K_{c_i}$ are compact, for any $i=1,\ldots,j$.

Now, we distinguish two situations. Either $\{c_i:\,\,i=1,\ldots,j\}$ are $j$ distinct critical values $\widehat{I}_{\lambda,\theta}$ or $c_n=c_{n+1}=\ldots=c_j=c$ for some $n\in\{1,\ldots,j-1\}$. In the second situation, since $K_c$ is compact, by a deformation lemma, Proposition \ref{s4p1} and arguing as in \cite[Lemma 3.6]{FFW}, we get
$$
\gamma(K_c)\geq j-n+1\geq2.
$$
Thus, $K_c$ has infinitely many points, see \cite[Remark 7.3]{Rabinowitz}, which are infinitely many critical values for $\widehat{I}_{\lambda,\theta}$ by Lemma \ref{l4.1}$(ii)$.

Consequently, $\widehat{I}_{\lambda,\theta}$ admits at least $j$ negative critical values, which represent at least $j$ negative critical values for $I_{\lambda,\theta}$ thanks to Lemma \ref{l4.1}$(i)$.
\end{proof}

\section{The existence of a nonnegative solution}\label{sec4}

We observe that since $P$ is a nonnegative nontrivial function satisfying $(P_{0})$, then  the set
$$\mathcal{A}_+:=\{\varphi\in E:\,\,\int_{\Omega}P(x)\varphi_+^kdx>0\}\neq\emptyset.$$
Indeed, using the hypothesis on $P$ and the Lebesgue differentiation theorem, there exists a point $x_0\in\Omega$ such that
$$\lim_{\varepsilon\to 0}\int_{B(x_0,\varepsilon)}P(x)dx=P(x_0)>0.$$ Therefore, there exists $\varepsilon>0$ such that $\int_{B(x_0,\varepsilon)}P(x)dx>0$.  Now choose $\varphi\in C^\infty_c(B(x_0,2\varepsilon))$ with $0\leq\varphi\leq 1$, $\varphi\equiv 1$ in $B(x_0,\varepsilon)$. Therefore, 
$$\int_{\Omega}P(x)\varphi^kdx=\int_{B(x_0,\varepsilon)}P(x)\varphi^kdx+\int_{B(x_0,2\varepsilon)\setminus B(x_0,\varepsilon)}P(x)\varphi^kdx\geq \int_{B(x_0,\varepsilon)}P(x)dx>0.$$
Hence, $\varphi\in \mathcal{A}_+$. Therefore, we can define 
\begin{equation}\label{alpha*}
\overline{\lambda}:=\inf_{\varphi\in\mathcal{A}_+}C_0(p,q,k)\left(\frac{[\varphi]^p_{s_1,p}}{p}\right)^\frac{k-q}{p-q}\left(\frac{[\varphi]^q_{s_2,q}}{q}\right)^\frac{p-k}{p-q}\frac{k}{\int_{\Omega}P(x)\varphi_+^kdx},
\end{equation}
where $C_0(p,q,k):=(\frac{p-q}{k-q})^\frac{k-q}{p-q}(\frac{p-q}{p-k})^\frac{p-k}{p-q}$.

Since we look for a nonnegative solution of \eqref{1.1}, we define the energy functional $J_{\lambda,\theta}:E\rightarrow \R$  such that
\begin{equation*}
J_{\lambda,\theta}(u)=\frac{1}{p}[u]_{s_{1},p}^{p}+\frac{1}{q}[u]_{s_{2},q}^{q}-\frac{\lambda}{k}\int_{\Omega}  P(x)(u_{+})^{k}dx-\frac{\theta}{p_{s_{1}}^{*}}\int_{\Omega}(u_{+})^{p_{s_{1}}^{*}} dx.
\end{equation*}
Then, dealing with $\overline{\lambda}$ given above, we study functional $J_{\lambda,\theta}$ on a suitable ball
$$
B_\varrho(0)=\left\{u\in E:\,\,\|u\|_E\leq \varrho\right\},
$$
proving the following result.

\begin{lemma}\label{negativo}
Let $\lambda>\overline{\lambda}$, with $\overline{\lambda}$ as defined in \eqref{alpha*}. Then, there exist $\varrho>0$, $\alpha>0$ and $\widehat{\theta}>0$ such that
$$\inf_{u\in B_\varrho(0)} J_{\lambda,\theta}(u)<0<\alpha\leq\inf_{u\in \partial B_\varrho(0)} J_{\lambda,\theta}(u),$$
whenever $\theta\in\left(0,\widehat{\theta}\right)$.
\end{lemma}

\begin{proof}
Fix $\lambda>\overline{\lambda}$. Thus, there exists $\varphi\in E$ such that 
\begin{equation}\label{7-5-24-1}
C_0(p,q,k)\left(\frac{[\varphi]^p_{s_1,p}}{p}\right)^\frac{k-q}{p-q}\left(\frac{[\varphi]^q_{s_2,q}}{q}\right)^\frac{p-k}{p-q}\frac{k}{\int_{\Omega}P(x)\varphi_+^kdx}<\lambda.
\end{equation}
Now, for any $t>0$, we have
$$J_{\lambda,\theta}(t\varphi)=t^qf_\lambda(t)-\frac{\theta t^{p_{s_{1}}^{*}}}{p_{s_{1}}^{*}}\int_{\Omega}\varphi_+^{p_{s_{1}}^{*}} dx,$$
where 
\begin{align*}
\begin{split}
&f_\lambda(t):=d_1+d_2t^{p-q}-\lambda d_3t^{k-q},\quad\mbox{with}\\
&d_{1}:=\frac{[\varphi]_{s_{2},q}^{q}}{q},\quad d_{2}:=\frac{[\varphi]_{s_{1},p}^{p}}{p}\quad\mbox{and}\quad
d_{3}:=\frac{1}{k}\int_{\Omega}  P(x)\varphi_+^{k}dx.
\end{split}
\end{align*}
We observe that there exists $t^*=t^*(\lambda)$ as
$$
t^*:=\left[\lambda\cdot\frac{(k-q)\,d_3}{(p-q)\,d_2}\right]^\frac{1}{p-k},
$$
such that
$$\min_{t\geq0}f_\lambda(t)=f_\lambda\left(t^*\right)=
d_1-\lambda^{\frac{p-q}{p-k}}\cdot\frac{p-k}{p-q}\cdot
\left(\frac{k-q}{p-q}\right)^\frac{k-q}{p-k}\cdot(d_2)^\frac{q-k}{p-k}\cdot(d_3)^\frac{p-q}{p-k}.$$
Thus, by \eqref{7-5-24-1} we get
$$\min_{t\geq0}f_\lambda(t)=f_\lambda\left(t^*\right)<0,$$
which yields
$$J_{\lambda,\theta}(t^*\varphi)=\left(t^*\right)^qf_\lambda(t^*)-\frac{\theta\left(t^*\right)^{p^*_{s_1}}}{p^*_{s_1}}\int_{\Omega}\varphi_+^{p^*_{s_1}}dx <0.$$
Let $\overline{C}_0:=\left[\displaystyle\frac{2p}{k}S^{-\frac{k}{p}}\|P\|_r\right]^{\frac{1}{p-k}}$, if we fix $\varrho=\varrho(\lambda)$ as
\begin{equation}\label{erre}
\varrho:=\max\left\{1+t^*\|\varphi\|_E+2\overline{C}_0\lambda^{\frac{1}{p-k}},
\left[\frac{\lambda q2^{q+1}}{k}S^{-\frac{k}{p}}\|P\|_r\left(\overline{C}_0\lambda^{\frac{1}{p-k}}\right)^k\right]^\frac{1}{q}\right\},
\end{equation}
clearly $\|t^*\varphi\|_E\leq \varrho$ implying $t^*\varphi\in B_\varrho(0)$.
Therefore, we get
$$\inf_{u\in B_\varrho(0)} J_{\lambda,\theta}(u)\leq J_{\lambda,\theta}\left(t^*\varphi\right)<0.$$

On the other hand, we claim that there exists $d>0$ such that
\begin{equation}\label{affermazione2}
\inf_{u\in \partial B_\varrho(0)} J_{\lambda,\theta}(u)\geq d,
\end{equation}
for $\theta$ sufficiently small.
By $(P_0)$ and \eqref{a2}, for any $u\in E$ we have
\begin{equation}\label{boh}
J_{\lambda,\theta}(u)\geq\frac{1}{p}[u]_{s_1,p}^p+\frac{1}{q}[u]_{s_2,q}^q-\lambda\frac{S^{-\frac{k}{p}}\|P\|_r}{k}[u]_{s_1,p}^k-
\frac{\theta S^{-\frac{p_{s_1}^*}{p}}}{p_{s_1}^*}[u]_{s_1,p}^{p_{s_1}^*}.
\end{equation}
Let $u\in E$ with $\|u\|_E=\varrho$. In order to prove \eqref{affermazione2}, we distinguish the following two cases.
\\

	\begin{itemize}
		\item \textbf{Case 1.} {\em Let $\displaystyle\frac{1}{p}[u]_{s_1,p}^p\leq \displaystyle\frac{2\lambda}{k}S^{-\frac{k}{p}}\|P\|_r[u]^k_{s_1,p}$, that is $[u]_{s_1,p}\leq \overline{C}_0\lambda^{\frac{1}{p-k}}$.}
		\\
		
		\noindent By \eqref{erre}, we have 
		$$[u]_{s_1,p}\leq \frac{1}{2}\varrho=\frac{1}{2}\|u\|_E=\frac{1}{2}\left([u]_{s_1,p}+[u]_{s_2,q}\right),$$
		from which $[u]_{s_1,p}\leq[u]_{s_2,q}$ and so
		$$
		2[u]_{s_2,q}\geq\|u\|_E.
		$$
		Furthermore, we deduce
		\begin{align*}
			\frac{\lambda}{k}S^{-\frac{k}{p}}\|P\|_r[u]^k_{s_1,p}
			\leq \frac{\lambda}{k}S^{-\frac{k}{p}}\|P\|_r\left(\overline{C}_0\lambda^{\frac{1}{p-k}}\right)^k
			\leq \frac{1}{q2^{q+1}}\varrho^q=\frac{1}{q2^{q+1}}\|u\|_E^q.
		\end{align*}
		Using the estimates above, we obtain from \eqref{boh} that
		\begin{align*}
			\notag	J_{\lambda,\theta}(u)\geq \frac{1}{q2^q}\|u\|_E^q-\frac{1}{q2^{q+1}}\|u\|_E^q-
\frac{\theta S^{-\frac{p_{s_1}^*}{p}}}{p_{s_1}^*}\|u\|_E^{p_{s_1}^*}
=\frac{S^{-\frac{p_{s_1}^*}{p}}\varrho^{p_{s_1}^*}}{p_{s_1}^*}\left(\frac{p_{s_1}^*}{q2^{q+1}S^{-\frac{p_{s_1}^*}{p}}}\varrho^{q-p_{s_1}^*}-\theta\right).
		\end{align*}
		Thus, setting
		$$
		\theta_a:=\frac{p_{s_1}^*}{q2^{q+1}S^{-\frac{p_{s_1}^*}{p}}}\varrho^{q-p_{s_1}^*}\quad\mbox{and}\quad
		\alpha_a:=\frac{S^{-\frac{p_{s_1}^*}{p}}\varrho^{p_{s_1}^*}}{p_{s_1}^*}\left(\frac{p_{s_1}^*}{q2^{q+1}S^{-\frac{p_{s_1}^*}{p}}}\varrho^{q-p_{s_1}^*}-\theta\right),
		$$
		we get \eqref{affermazione2} with $d=\alpha_a$, whenever $\theta\in(0,\theta_a)$.
\\
		
\item \textbf{Case 2.} {\em Let $\displaystyle\frac{1}{p}[u]_{s_1,p}^p> \displaystyle\frac{2\lambda}{k}S^{-\frac{k}{p}}\|P\|_r[u]^k_{s_1,p}$, that is $[u]_{s_1,p}> \overline{C}_0\lambda^{\frac{1}{p-k}}$.}
		\\
		
		\noindent In this case, we derive from \eqref{boh} that\begin{align*}
			J_{\lambda,\theta}(u)&\geq \frac{1}{2p}[u]_{s_1,p}^p+\frac{1}{q}[u]_{s_2,q}^q-
\frac{\theta S^{-\frac{p_{s_1}^*}{p}}}{p_{s_1}^*}\|u\|_E^{p_{s_1}^*}\\
&\geq\frac{1}{2p}\left(\overline{C}_0^{p-q}\lambda^{\frac{p-q}{p-k}}[u]_{s_1,p}^q+[u]_{s_2,q}^q\right)-
\frac{\theta S^{-\frac{p_{s_1}^*}{p}}}{p_{s_1}^*}\|u\|_E^{p_{s_1}^*}\\
&\geq\frac{\overline{C}_1}{2^qp}\|u\|_E^q-\frac{\theta S^{-\frac{p_{s_1}^*}{p}}}{p_{s_1}^*}\|u\|_E^{p_{s_1}^*}\\
&=\frac{S^{-\frac{p_{s_1}^*}{p}}\varrho^{p_{s_1}^*}}{p_{s_1}^*}
\left(\frac{\overline{C}_1p_{s_1}^*}{2^qpS^{-\frac{p_{s_1}^*}{p}}}\varrho^{q-p_{s_1}^*}-\theta\right),
		\end{align*}
with $\overline{C}_1:=\overline{C}_1(\lambda)=\min\left\{\overline{C}_0^{\,p-q}\lambda^{\frac{p-q}{p-k}},1\right\}$.
Hence, setting
		$$
		\theta_b:=\frac{\overline{C}_1p_{s_1}^*}{2^qpS^{-\frac{p_{s_1}^*}{p}}}\varrho^{q-p_{s_1}^*}\quad\mbox{and}\quad
		\alpha_b:=\frac{S^{-\frac{p_{s_1}^*}{p}}\varrho^{p_{s_1}^*}}{p_{s_1}^*}
\left(\frac{\overline{C}_1p_{s_1}^*}{2^qpS^{-\frac{p_{s_1}^*}{p}}}\varrho^{q-p_{s_1}^*}-\theta\right),
		$$
		we get \eqref{affermazione2} with $d=\alpha_b$, whenever $\theta\in(0,\theta_b)$.
	\end{itemize}

\vspace{0.2cm}
\noindent
It is enough to set $\widehat{\theta}:=\min\{\theta_a,\theta_b\}$ and $\alpha:=\min\{\alpha_a,\alpha_b\}$ in order to complete the proof.
\end{proof}

\begin{proof}[{\bf Proof of the Theorem \ref{t3}}]

We first observe that $\overline{c}$ given in \eqref{cappello} is positive whenever $\theta<\widetilde{\theta}$, with $\widetilde{\theta}=\widetilde{\theta}(\lambda)$ set as
\begin{equation}\label{qtilde}
\widetilde{\theta}:=\left[
\frac{S^{\frac{p_{s_1}^*}{p_{s_1}^*-p}}}{\left(\lambda\|P\|_r\right)^{\frac{p_{s_1}^*}{p_{s_1}^*-k}}}
\cdot\frac{k}{p_{s_1}^*-k}
\cdot\left(\frac{p_{s_1}^*-p}{p-k}\right)^{\frac{p_{s_1}^*}{p_{s_1}^*-k}}
\right]^\frac{(p_{s_1}^*-p)(p_{s_1}^*-k)}{p_{s_1}^*(p-k)}.
\end{equation}
Thus, let us fix $\lambda>\overline{\lambda}$, with $\overline{\lambda}$ given in \eqref{alpha*}. Let us fix $\theta\in\left(0,\theta^*\right)$ with $\theta^*:=\min\left\{\widehat{\theta},\widetilde{\theta}\right\}$, where $\widehat{\theta}$ and $\widetilde{\theta}$ are as in Lemma~\ref{negativo} and \eqref{qtilde} respectively.

Now, let us set the minimum $$c_\varrho:=\inf_{u\in B_\varrho(0)} J_{\lambda,\theta}(u).$$
Thanks to Lemma \ref{negativo}, we can apply the Ekeland variational principle to $J_{\lambda,\theta}$ which provides a minimizing sequence $\{u_n\}_n\subset B_\varrho(0)$, such that  $J_{\lambda,\theta}(u_{n})\rightarrow c_\varrho$ and $J'_{\lambda,\theta}(u_{n})\rightarrow 0$.
However, since $c_\varrho<0<\overline{c}$ by Lemma \ref{negativo}, we can apply Lemma \ref{l5} for the sequence $\{u_n\}_n$, so that there exists $u\in E$ such that $u_n\to u$. Hence $J'_{\lambda,\theta}(u)=0$ and $J_{\lambda,\theta}(u)=c_\varrho<0$. Thus, $u$ is a nontrivial weak solution of the problem \eqref{1.1} with negative energy.  Next, we prove that $u$ is nonnegative. Indeed, we have
\begin{equation*}
	\begin{split}
		0=\langle J'_{\lambda,\theta}(u),u_{-}\rangle=
		&\iint_{\R^{2N}} \dfrac{|u(x)-u(y)|^{p-2}(u(x)-u(y))(u_{-}(x)-u_{-}(y))}{|x-y|^{N+ps_1}}dx dy\\
		&\quad+\iint_{\R^{2N}} \dfrac{|u(x)-u(y)|^{q-2}(u(x)-u(y))(u_{-}(x)-u_{-}(y))}{|x-y|^{N+qs_2}}dx dy\\
		& \geq\iint_{\R^{2N}} \dfrac{|u_{-}(x)-u_{-}(y)|^{p}}{|x-y|^{N+ps_1}}dx dy+\iint_{\R^{2N}} \dfrac{|u_{-}(x)-u_{-}(y)|^{q}}{|x-y|^{N+qs_2}}dx dy\\
		&= [u_{-}]_{s_{1},p}^{p}+[u_{-}]_{s_{2},q}^{q},
	\end{split}
\end{equation*}
which implies that $u_{-}=0$. Hence, $u\geq 0$. This completes the proof.
\end{proof}

\vspace{5mm}

\section*{Declaration}

{\bf Funding.} M.~Bhakta and S.~Gupta are partially supported by the {\em  DST Swarnajaynti fellowship}  (SB/SJF/2021-22/09).
A. Fiscella is member of {\em Gruppo Nazionale per l'Analisi
Ma\-te\-ma\-ti\-ca, la Probabilit\`a e le loro Applicazioni}
(GNAMPA) of the {\em Istituto Nazionale di Alta Matematica} (INdAM).
A. Fiscella realized the manuscript within the auspices of the FAPESP project titled
\textit{Non-uniformly elliptic problems} (2024/04156-0).

\vspace{5mm}

{\bf Competing interests.} The authors have no competing interests to declare that are relevant to the content of this article.

\vspace{5mm}

{\bf Data availability statement.} Data sharing not applicable to this article as no data sets were generated or analysed during the current study.

\end{document}